\documentclass[11pt,reqno]{amsart}
\usepackage{amscd,amssymb,amsmath,amsthm}
\usepackage[arrow,matrix]{xy}
\usepackage{graphicx}
\usepackage{epstopdf}
\usepackage{color}
\newtheorem{thm}{Theorem}
\newtheorem{defn}{Definition}

\newtheorem{pro}{Proposition}

\newtheorem{cor}{Corollary}

\newtheorem{ex}{Example}

\numberwithin{equation}{section} 

\begin{document}
\title [Some remarks on evolution algebras corresponding to permutations.]
{Some remarks on evolution algebras corresponding to permutations.}

\author{B. A. Narkuziyev}

\address{Institute of Mathematics, Tashkent, 100125, Uzbekistan.}
 \email { bnarkuziev@yandex.ru}

\begin{abstract} In the present paper we  describe  absolute nilpotent and some idempotent elements of an $n$- dimensional evolution algebra corresponding to two permutations and we decompose such algebras to the direct sum of evolution algebras corresponding to cycles of the permutations.
\end{abstract}


 \maketitle

\section{Introduction}

An evolution algebra is an abstract system, it gives an insight for the study of non-Mendelian genetics. Results on genetic evolution and genetic algebras, can be found in Lyubich's book \cite{Lyu1992}.
In Tian's book \cite{Tian2008} the foundations of evolution algebras are developed and several basic properties are studied. The concept of evolution algebra lies between algebras and dynamical systems. Algebraically, evolution algebras are non-associative Banach algebras; dynamically, they represent discrete dynamical systems. Evolution algebras have many connections with other mathematical fields including graph theory, group theory, Markov chains, dynamic systems, stochastic processes, mathematical physics, etc.\cite{CasLad2011}-\cite{Lyu1992}. Rozikov and Tian \cite{Roz2011} studied algebraic structures evolution algebras associated with Gibbs measures defined on some graphs. In \cite{CasLad2011}, \cite{CasLad2013} some properties of chain of evolution algebras, dibaricity  of evolution algebras, a criteria for an evolution algebra to be baric were studied. In \cite{KhudOmirov2015}, \cite{Nar2014} some properties of evolution algebra corresponding to a permutation  were studied. In this paper we study evolution algebras corresponding to two permutations.

The paper is organized as follows. In  Section 2 we give main definitions and some properties of  $n$- dimensional evolution algebra corresponding to two permutations. Therein  we also  give a criterion for this algebra  to be baric. In Section 3 we obtain a necessary and sufficient conditions to have unique trivial absolute nilpotent element for any $n$-dimensional
evolution algebras with rank of matrix of structural constants is $n-2$ and absolute nilpotent elements of
 evolution algebras corresponding to two  permutations are fully studied. In Section 4 we study idempotent elements of the  evolution algebra  and give a general  analysis of idempotent elements of the two- dimensional evolution algebra corresponding to permutations. Section 5 is devoted to the isomorphisms of evolution algebras corresponding to  permutations.

\section{A criterion for an  $n$- dimensional evolution algebra corresponding to two permutations to be baric.}

\begin{defn} Let $(E,\cdot)$ be an  algebra  over  a  field  $K$. If it admits a basis $e_1,e_2...,$ such that
 $$e_{i}\cdot e_{j}=0, \ \ \ \ \ \ \ \ \ \ if \ \ \ i\neq j \ ; $$
$$
\ e_{i}\cdot e_{i}=\sum\limits_{k} a_{ik}e_k \ ,\ \ \ \ \ \ \ \ \ \
 for\ \ \ any\ \ \ \ i\ ,
$$
 then this algebra is called an $evolution \  algebra$. This basis is  called natural basis.
\end{defn}
We denote by $M=(a_{ij})$  the  matrix  of    the structural constants of the evolution algebra  $E$.

Let $S_{n}$ be  the group  of   permutations  of  degree   $n$. Take $\pi,\tau \in S_n,$

$$ \pi=\left( \begin{array}{cc} 1 \,\,\,\ \ \ \ 2 \ \ \ \,\,\,\ 3 \ \ \ \ \ \ \ ...\ \ \ \,\,\,\ n \\
\pi(1) \,\ \pi(2) \,\ \pi(3) \ \ \  ... \,\ \ \ \pi(n)\end{array} \right),\ \ \ $$

$$ \ \ \ \ \ \ \ \ \ \ \ \ \ \ \ \ \ \ \ \tau=\left( \begin{array}{cc} 1 \,\,\,\ \ \ \ 2 \ \ \ \,\,\,\ 3 \ \ \ \ \ \ \ ...\ \ \ \,\,\,\ n \\
\tau(1) \,\ \tau(2) \,\ \tau(3) \ \ \  ... \,\ \ \ \tau(n)\end{array} \right), \ \ \ \pi(i), \tau(i)\in \{1,...,n\}.$$

Consider a $n$-dimensional evolution  algebra  $E^{n}_{\pi,\tau}$ over  the  field   $R$  with  a  finite  natural  basis   $e_1,e_2,...e_n,$ and
    multiplication   given  by:
    $$e_{i}\cdot e_{j}=0, \ \  if \ \  i\neq j \ ;\ \ \ \ $$
 $$  e_{i}\cdot e_{i}=a_{i\pi(i)}\cdot e_{\pi(i)} + a_{i\tau(i)}\cdot e_{\tau(i)},\ \ \
 for\ \ \ any\ \ \ \ i\ ,$$
then this algebra  is  called  an  \emph{evolution algebra corresponding to permutations} $\pi$ and $\tau$.

\
We assume that $ \pi\neq\tau.$ For the case $ \pi=\tau$ some properties of the algebra are known \cite{KhudOmirov2015}-\cite{Nar2014}.

The following properties  of  evolution algebras are known \cite{Tian2008}:

\begin {itemize}
\item[(1)] Evolution algebras are not associative, in general.

\item[(2)] Evolution algebras are commutative, flexible.

\item[(3)] Evolution algebras are not power-associative, in general.

\item[(4)] The direct sum of evolution algebras is also an evolution algebra.

\item[(5)] The Kronecker product of evolution algebras is an evolution algebra.
\end {itemize}
These  properties  are  hold  for  $E^{n}_{\pi,\tau}$   too.

A  $character$  for  an  algebra   $\mathcal A$  is  a  nonzero  multiplicative  linear form on  $\mathcal A$ , that is,
 a   nonzero  algebra  homomorphism  from   $\mathcal A$  to $R$ \cite{Lyu1992} .

\begin{defn}  A   pair  $( \mathcal A ,\sigma)$  consisting \  of \  an\   algebra \  $\mathcal A$ \  and \  a \  character $\sigma$
 on \  $\mathcal A$  is  called   a\   baric\  algebra.\  The \ homomorphism $\sigma$  is\  called\   the\   weight (or baric) function
  of  $\mathcal A$  and  $\sigma(x)$ the weight (baric value) of  $x$.
\end{defn}

\begin{thm}\label{thm1}
An  evolution algebra $ E^{n}_{\pi,\tau }$  over the field  R, is baric if and only if one of the following conditions holds:
\begin {itemize}
\item [1)] $k_{0}$ is a fixed point for $\pi$ (or $\tau $),$ \ \tau({k_{0}})\neq k_{0} \ (or\  \pi(k_{0})\neq k_{0})$  and  its  matrix  $M=(a_{ij})\ \ i,j=1,...,n$ of structural constants satisfies  $\ \  a_{k_{0}k_{0}}\neq0$, $a_{\tau^{-1}(k_{0})k_{0}}=0$
     $\   (or \  a_{\pi^{-1}(k_{0})k_{0}}=0)$.
  In this case the corresponding weight function is $ \sigma(x)=a_{k_{0}k_{0}} x_{k_{0}}.$

 \item[2)]   $k_{0}$ is a fixed point for both $\pi$ and $\tau$, and $a_{k_{0}k_{0}}\neq0.$ In this case the corresponding weight function is $ \sigma(x)=2 a_{k_{0}k_{0}} x_{k_{0}}.$
\end {itemize}
\end{thm}
\begin{proof}
Follows from Theorem 3.2  in  \cite{CasLad2011}, because there is only one element $a_{k_{0}k_{0}}$
(or $2\cdot a_{k_{0}k_{0}}$) which is not equal to zero in the   $ k_{0}^{th}$ column.
\end{proof}

An  evolution algebra $ E^{n}_{\pi,\tau }$ may have several weight functions. As a corollary of Theorem \ref{thm1} we have
\begin{cor}
If the permutations $\pi$ and $ \tau$, mentioned in Theorem \ref{thm1} have $m$ fixed points $k_{m}, \ m\leq n$ which satisfy conditions of Theorem \ref{thm1}, then the evolution algebra $ E^{n}_{\pi,\tau }$ has exactly m weight functions $ \sigma(x)=a_{k_{m}k_{m}}\cdot x_{k_{m}} $ (or $ \sigma(x)=2a_{k_{m}k_{m}}\cdot x_{k_{m}}$ ).
\end{cor}

\section{Absolute nilpotent elements.}

\begin{defn}
An element $x$ of an algebra $\mathcal A$ \ is called an $absolute \ nilpotent$ if $x^{2}=0$.
\end{defn}

Let $E=R^{n}$ be an evolution algebra over the  field $R$ with the matrix
    $M=(a_{ij})$ of structural constants, then for any $x=\sum\limits_{i}x_{i}e_{i}$ and
   $y=\sum\limits_{i}y_{i}e_{i}\in R^{n}$ we have
    $$x y=\sum\limits_{i}x_{i}y_{i}e_{i}e_{i}=\sum\limits_{j}( \sum\limits_{i}a_{ij}x_{i}y_{i} )e_{j} ,\\\\\ x^{2}=\sum\limits_{j}(\sum\limits_{i} a_{ij} x^{2}_{i})e_{j}.$$

For a $n$-dimensional  evolution algebra $ E $ over the field $R$ consider operator $V:R^{n}\rightarrow R^{n},\ x\mapsto V(x)=x' $ defined as \
$x'_{j}=\sum\limits_{i=1}^{n}a_{ij} x^{2}_{i}, \   j=1,...,n.$ This operator is called $evolution \  operator$  \cite{Lyu1992} .\\
We have $V(x)=x^{2}$, hence the equation $V(x)=x^{2}=0$ is given by the following system
\begin{equation} \label{3.1}
\sum\limits_{i=1}^{n}a_{ij}x^{2}_{i}=0,\ \ \ \ \ \ j=1,...,n.
\end{equation}
In this section we consider absolute nilpotent elements of evolution algebras and evolution algebra corresponding to two permutations( see \cite{CasLad2011},  \cite{Tian2008},\cite {KhudOmirov2015},\cite {Nar2014} ).

Let $M^{T}$ is transposed matrix of $M$ then we know that $det (M)=det (M^{T})$ therefore   if $det(M)\neq 0$ then the system $(3.1)$ has the unique solution $(0,...,0).$ If $det(M)= 0$ and $rank(M)=r$ then we can assume that the first $r$ rows of $M$
are linearly independent, consequently, the system (\ref{3.1}) can be written as
\begin{equation} \label{3.2}
 x^{2}_{i}=-\sum \limits_{j=r+1}^{n}d_{ij} x^{2}_{j},\ \ \  i=1,...,r.
 \end{equation}

where $d_{ij}=\frac{det(M_{ij}^{T})}{det(M_{r})}$ with $ M_{r}=(a_{ik})_{i,k=1,...,r},  $
$$ M_{ij}^{T}= \left( \begin{array}{cccc} a_{11}\ ... \ a_{i-1,1}\ a_{j1} \ a_{i+1,1} \ ...  \ a_{r1}
    \\ a_{12}\ ... \ a_{i-1,2}\ a_{j2} \ a_{i+1,2} \ ...  \ a_{r2} \\ ...\ \ \ \ \ \ ... \ \ \ \ \ \ ... \ \ \ \ \ \ \ .... \\
 a_{1r}\ ... \ a_{i-1,r}\ a_{jr} \ a_{i+1,r} \ ...  \ a_{rr} \end{array} \right) \  . $$

An interesting problem is to find a necessary and sufficient condition on matrix $ D=(d_{ij})_{i=1,...,r ;  \ j=r+1,...,n }\ \ $ under which the system (\ref{3.2}) has a unique solution. The difficulty of the problem depends on rank $r$. In \cite{CasLad2011} the case $rank(M)=n-1 $ is well studied. That is why  here we shall
consider the case $rank(M)=n-2 $. In this case the system of eq.(\ref{3.2}) will be
\begin{equation}\label{3.3}
 x^{2}_{i}=-(d_{i,n-1}\cdot x^{2}_{n-1}+ d_{i,n}\cdot x^{2}_{n}),\ \ i=1,...,n-2.
 \end{equation}
The following theorem identifies the necessary and sufficient conditions for the system (\ref{3.3}) to have a unique  trivial solution.
\begin{thm}\label{thm2}
Let  $rank(M)=n-2$. The evolution algebra $E=R^{n}$ has a unique absolute nilpotent element $(0,...,0)$ if and only if one of the following conditions is satisfied:
\begin {itemize}
 \item[(i)]  $  \left\{\begin{array}{ll} d_{i_{0},n-1}>0 \\ d_{i_{0},n}>0\end{array} \right.$ for some $i_{0}\in \{1,...,n-2\}$ ;

 \item[(ii)] $ \left\{\begin{array}{lll} d_{k,n-1}>0 \\ d_{k,n}=0 \\ d_{m,n}>0
 \end{array} \right. \  or  \  \left\{\begin{array}{lll} d_{k,n-1}=0 \\ d_{k,n}>0 \\ d_{m,n-1}>0
 \end{array} \right. $ for some  $ k, m \in \{1,...,n-2\}$ ;

 \item[(iii)]  $  \left\{\begin{array}{ll} d_{s,n-1}> 0 \\ d_{s,n}< 0\end{array} \right. $
  and $  \left\{\begin{array}{ll} d_{t,n-1}< 0 \\ d_{t,n}> 0\end{array} \right. $
  for some $ s, t \in \{1,...,n-2 \}$
   and $$ \max_t \{ {\frac{-d_{t,n}}{d_{t,n-1}}}:d_{t,n-1}< 0,\ d_{t,n}> 0 \}>  \min_s \{ {\frac{-d_{s,n}}{d_{s,n-1}}}: d_{s,n-1}> 0 ,\ d_{s,n}< 0 \}.$$
\end {itemize}
\end{thm}
\begin{proof}(\textit{Necessity})
 It is easy to see that if the conditions (i) and (ii) are satisfied then $x_{n-1}=x_{n}=0$. It means that in these cases the system of equations (\ref{3.3}) has a unique trivial solution.

For the system (\ref{3.3}) to have  a  solution, it must satisfy the next inequality
$$d_{i,n-1}\cdot x^{2}_{n-1}+ d_{i,n}\cdot x^{2}_{n}\leq 0.$$
From this inequality we have
 $$|x_{n-1}|\leq \min_s\{\frac{-d_{s,n}}{d_{s,n-1}} \}\cdot|x_{n}| $$
  and $$ |x_{n-1}|\geq \max_t\{\frac{-d_{t,n}}{d_{t,n-1}} \}\cdot|x_{n}|$$
 for  all $s$ and $ t$ which satisfy the condition $(iii)$. Therefore
 \begin{equation}\label{3.4}
 \max_t\{\frac{-d_{t,n}}{d_{t,n-1}} \}\cdot|x_{n}| \leq |x_{n-1}|\leq \min_s\{\frac{-d_{s,n}}{d_{s,n-1}} \}\cdot|x_{n}|,
 \end{equation}
 this shows that if
 $$ \max_t \{ {\frac{-d_{t,n}}{d_{t,n-1}}}\}>  \min_s \{ {\frac{-d_{s,n}}{d_{s,n-1}}}\}$$
 then the system (\ref{3.3}) has  a unique trivial solution only.\\
(\textit{Sufficiency}). Assume that none of the conditions of the  theorem are  satisfied  then it is enough to consider the following cases:
\begin {itemize}
 \item[(a)] $ \left\{\begin{array}{ll} d_{i,n-1}\leq 0 \\ d_{i,n}\leq 0\end{array} \right. $ for all $i\in {1,...,n-2} $ ;

 \item [(b)] $ \left\{\begin{array}{ll} d_{k,n-1}> 0 \\ d_{k,n}=0\end{array} \right. $ for some $k$ and $d_{i,n}\leq 0 $ for all
  $i\neq k$ ,where $i,k\in \{1,...,n-2\} $;

 \item[(c)] $ \left\{\begin{array}{ll} d_{k,n-1}= 0 \\ d_{k,n}>0\end{array} \right. $ for some $k$ and $d_{i,n-1}\leq 0 $ for all
  $i\neq k$ ,where $i,k\in \{1,...,n-2\} $,

 \item[(d)] $ \left\{\begin{array}{ll} d_{i,n-1}> 0 \\ d_{i,n}<0\end{array} \right. $ or $ \left\{\begin{array}{ll} d_{i,n-1}< 0 \\ d_{i,n}>0\end{array} \right.$
 for all $i\in \{1,...,n-2\}$
 \item[(e)] $  \left\{\begin{array}{ll} d_{s,n-1}> 0 \\ d_{s,n}< 0\end{array} \right. $
  and $  \left\{\begin{array}{ll} d_{t,n-1}< 0 \\ d_{t,n}> 0\end{array} \right. $
  for some $s,t \in \{1,...,n-2 \}$ and
  $$ \max_t \{ {\frac{-d_{t,n}}{d_{t,n-1}}}:d_{t,n-1}< 0,\ d_{t,n}> 0 \} \leq  \min_s \{ {\frac{-d_{s,n}}{d_{s,n-1}}}: d_{s,n-1}> 0 ,\ d_{s,n}< 0 \}.$$
\end {itemize}
 In the cases (a),(b),(c) it is easy to find non-trivial solutions of the system of equations (\ref{3.3}) by selecting  $ x_{n-1}$ and $x_{n} $.

 (d) If $ \left\{\begin{array}{ll} d_{i,n-1}> 0 \\ d_{i,n}<0\end{array} \right.$ for all $i\in \{1,...,n-2\}$ then
  we can choose $ x_{n-1}$ and $x_{n} $  to satisfy  $|x_{n-1}|\leq min \{\frac{-d_{i,n}}{d_{i,n-1}}\}\cdot |x_{n}|$  and if $ \left\{\begin{array}{ll} d_{i,n-1}< 0 \\ d_{i,n}>0\end{array} \right.$ for all $i\in \{1,...,n-2\}$ then  we can select $ x_{n-1}$ and $x_{n} $  to satisfy  $|x_{n-1}|\geq max \{\frac{-d_{i,n}}{d_{i,n-1}}\}\cdot |x_{n}|$   to form non-trivial  solutions of the system (\ref{3.3}).

 (e) If $  \left\{\begin{array}{ll} d_{s,n-1}> 0 \\ d_{s,n}< 0\end{array} \right. $
  and $  \left\{\begin{array}{ll} d_{t,n-1}< 0 \\ d_{t,n}> 0\end{array} \right. $
  for some $s,t \in \{1,...,n-2 \}$ then in order for the system (\ref{3.3}) to have non-trivial solutions the  inequality (\ref{3.4}) must be performed.

  Hence the non-fulfillment of any of the conditions of Theorem \ref{thm2} means that the system (\ref{3.3}) has non-trivial solutions.
\end{proof}

The above theorem holds  for any $n$- dimensional evolution algebra $ E $ over the field $R$. Now we consider the  evolution algebra $ E^{n}_{\pi,\tau }$  over the field $ R$.
\\ Let  $x=\sum\limits_{i}x_{i}e_{i}\in R^{n}$ and $ x^{2}=\sum\limits_{i}x^{2}_{i}e_{i}e_{i}$ , for $ E^{n}_{\pi,\tau }$ we have
   \begin{equation}\label{3.5}
    x^{2}=\sum \limits_{i=1}^{n}x^{2}_{i}(a_{i\pi(i)}\cdot e_{\pi(i)} + a_{i\tau(i)}\cdot e_{\tau(i)}).
    \end{equation}
   Let $j_{k}=\tau^{-1}(\pi(k)),$ then $\sum \limits_{i}a_{i\tau(i)} x^{2}_{i}  e_{\tau(i)}=\sum \limits_{k} a_{j_{k}\pi(k)}x^2_{j_{k}}e_{\pi(k)} $  therefore from (\ref{3.5}) we have
   $$ x^{2}=\sum \limits_{k=1}^{n}e_{\pi(k)}(a_{k\pi(k)}x^2_{k}+ a_{j_{k}\pi(k)}x^2_{j_{k}} ). $$
   Thus the equation $x^{2}=0$ will be
    \begin{equation}\label{3.6}
     a_{k\pi(k)}x^2_{k}+ a_{j_{k}\pi(k)}x^2_{j_{k}}=0 ,\ \  k\in \{1,...,n\}.
    \end{equation}
   We first consider the question when the system (\ref{3.6}) has a unique trivial solution. The following proposition answers this question (but not completely, the full answer will be considered later ). Also in the next proposition for the case $rank(M)=n-1 $,  without loss of generality we can assume that the first $n-1$ rows of $M$ are linearly independent.
\begin{pro}
  Let $M$ be the matrix of structural constants for $ E^{n}_{\pi,\tau }$. The finite dimensional evolution algebra $ E^{n}_{\pi,\tau }$ has the unique absolute nilpotent element $(0,...,0)$
  if one of the following conditions is satisfied.
 \begin {itemize}
 \item [(i)] $det(M)\neq 0$ ;
  \item[(ii)]  $rank(M)=n-1 $ and  $det(M_{i_{0}n}^{T})\cdot  det(M_{n-1})>0$  for some $i_{0}\in \{1,...,n-1\}$;
 \item[ (iii)] For all $k\in \{1,...,n \}, \ a_{k\pi(k)}\cdot a_{j_{k}\pi(k)}>0 $ and in this case $rank(M)$ doesn't matter.
\end {itemize}
\end{pro}
\begin{proof}  The first and second parts of the proposition follow from  Proposition 4.12 in \cite {CasLad2011}.
Proof of part 3 comes from (\ref{3.6}) easily.
\end{proof}
The conditions (ii) and (iii) are various. To see  this difference between conditions we give the  following examples.
\begin{ex}
(For the case (ii) ) Let
 $$ \pi=\left( \begin{array}{cc} 1 \ \ \ 2 \ \ \ 3 \ \ \ 4 \\
 3 \ \ \ 1 \ \ \ 4 \ \ \ 2\end{array} \right),\ \ \tau=\left( \begin{array}{cc} 1 \ \ \ 2 \ \ \ 3 \ \ \ 4 \\
 2 \ \ \ 4 \ \ \ 3 \ \ \ 1\end{array} \right),\ \ \ $$ from the system (\ref{3.6}) we have the following  system of equations
 $$ \left\{\begin{array}{llll} a_{13}x^2_{1}+ a_{33}x^2_{3}=0 \\
a_{21}x^2_{2}+ a_{41}x^2_{4}=0 \\
a_{34}x^2_{3}+ a_{24}x^2_{2}=0  \\
a_{42}x^2_{4}+ a_{12}x^2_{1}=0 \
\end{array} \right. \ \ \ and \ \ \ M = \left( \begin{array}{cccc}   0 \ \ \ \   a_{12} \ \ \  a_{13} \ \ \  0\\
   a_{21} \ \ \ \ \ 0  \ \ \  \ \ 0 \  \ \ \   a_{24} \\
  \ \ 0  \ \ \ \ \ \  0 \ \ \ \ a_{33} \ \ \  a_{34} \\
   a_{41} \ \ \  a_{42} \ \ \ \  0 \ \ \ \  0 \end{array} \right) \   $$
$ det M=a_{12}\cdot a_{24}\cdot a_{33}\cdot a_{41}-a_{13}\cdot a_{21}\cdot a_{34}\cdot a_{42}.$
If $a_{13}=a_{24}=-1$ and the other $a_{ij}=1$, then
$ det M=0$, $rank M=3 $,
 $M_{n-1}=M_{3}= \left( \begin{array}{ccc}  0 \ \ \ \ 1  \ -   1  \\
 1 \ \ \ \  0  \ \ \  \ \ 0   \\
  0 \ \ \ \ 0  \ \ \  \ \ 1    \end{array} \right),$ \
  $M_{i_{0}n}^{T}=M_{14}^{T}=  \left( \begin{array}{ccc}  1 \ \ \  1  \ \ \   0  \\
 1  \ \ \   0  \ \ \   0   \\
 0  \ \ \  0  \ \ \   1   \end{array} \right),$ \ \ \ \
 $ det M_{3} \cdot det M_{14}=(-1) \cdot (-1)=1$ .
 It is easy to see that the system $ \left\{\begin{array}{llll} -x^2_{1}+ x^2_{3}=0 \\
\ \ x^2_{2}+ x^2_{4}=0 \\
\ \ x^2_{3}- x^2_{2}=0  \\
\ \ x^2_{4}+ x^2_{1}=0 \
\end{array} \right.$ has only trivial solution.\\
\end{ex}
\begin{ex} \ (For the case (iii)) Let
 $$ \pi=\left( \begin{array}{cc} 1 \ \ \ 2 \ \ \ 3 \ \ \ 4 \\
 1 \ \ \ 3 \ \ \ 4 \ \ \ 2\end{array} \right),\ \ \tau=\left( \begin{array}{cc} 1 \ \ \ 2 \ \ \ 3 \ \ \ 4 \\
 4 \ \ \ 2 \ \ \ 1 \ \ \ 3\end{array} \right),\ \ \ $$ from the system (\ref{3.6}) we have the following  system of equations
 $$ \left\{\begin{array}{llll} a_{11}x^2_{1}+ a_{31}x^2_{3}=0 \\
a_{23}x^2_{2}+ a_{43}x^2_{4}=0 \\
a_{34}x^2_{3}+ a_{14}x^2_{1}=0  \\
a_{42}x^2_{4}+ a_{22}x^2_{2}=0 \
\end{array} \right.  \ \ \ and \ \ \  M = \left( \begin{array}{cccc} \ a_{11} \ \ \   0 \ \ \ \  0 \ \ \ \ \ a_{14}\\
   0 \ \ \ \  a_{22} \ \  a_{23} \  \ \  0  \\
  \ a_{31} \ \ \   0 \ \ \ \  0 \ \ \ \ \ a_{34}  \\
  0 \ \ \ \  a_{42} \ \   a_{43} \ \ \ 0 \end{array} \right) \   $$
if all $a_{ij}=1$ then $det M =0,\ rank M =2$ and this system has only a trivial solution.
\end{ex}
Now we consider the absolute nilpotent elements of $ E^{n}_{\pi,\tau }$ more clearly.\\
  Let $$ \tau^{-1}=\left( \begin{array}{cc} 1 \,\,\,\ \ \ \ \ \  \ 2 \ \ \  \  \ \,\, ...\ \ \ \ \ \,\,\ n \\
 \tau^{-1}(1) \,  \ \tau^{-1}(2) \, \ \  ... \,\  \tau^{-1}(n)\end{array} \right),\  $$ and
 $$ \tau^{-1}\circ \pi =\left( \begin{array}{cc} 1 \,\,\,\ \ \ \ 2  \ \ \ \ \ \ ...\ \ \ \,\,\,\ n \\
j_{1} \,\ \ \ \  j_{2} \,\  \ \ \  ... \,\ \ \ \ \ \ j_{n}\end{array} \right),\ \ \ \ \ where \ \ j_{k}=\tau^{-1}(\pi(k)).$$
Let $ \tau^{-1}\circ \pi $ be decomposed into product of independent cycles and $(l_{1}\ l_{2}\ ...\ l_{p})$ be one of its  cycles.
Finding the absolute  nilpotent elements of $ E^{n}_{\pi,\tau }$ means solving the system of equations (\ref{3.6}).
 From the system of equations (\ref{3.6}) we have
\begin{equation}\label{3.7}
 a_{l_{k}\pi(l_{k})}x^2_{l_{k}}+ a_{j_{l_{k}}\pi(l_{k})}x^2_{j_{l_{k}}}=0 .
 \end{equation}
 We know that $j_{k}=\tau^{-1}(\pi(k)) $, from this $j_{l_{k}}=\tau^{-1}\circ \pi(l_{k})$. So
   $$l_{k+1}=j_{l_{k}}=\tau^{-1}\circ \pi(l_{k}),\ \ \tau(l_{k+1})=\pi(l_{k})\ \        $$
   after using these equalities the  equation (\ref{3.7}) will be
   $$ a_{l_{k}\pi(l_{k})}x^2_{l_{k}} + a_{l_{k+1}\tau(l_{k+1})}x^2_{l_{k+1}}=0. $$
Since $l_{p+1}=l_{1}$, we form the following system
\begin{equation}\label{3.8}
 \left\{\begin{array}{llll} a_{l_{1}\pi(l_{1})}x^2_{l_{1}} \ \   + \ \    a_{l_{2}\tau(l_{2})}x^2_{l_{2}}=0 \\
a_{l_{2}\pi(l_{2})}x^2_{l_{2}} \ \  +\ \      a_{l_{3}\tau(l_{3})}x^2_{l_{3}}=0 \\
....................................... \\
a_{l_{p-1}\pi(l_{p-1})}x^2_{l_{p-1}} + a_{l_{p}\tau(l_{p})}x^2_{l_{p}}=0 \\
a_{l_{p}\pi(l_{p})}x^2_{l_{p}} \ \   + \ \  a_{l_{1}\tau(l_{1})}x^2_{l_{1}}=0 .\ \
 \end{array} \right.
 \end{equation}
 All equations of the system (\ref{3.6}) with unknowns $x_{l_{1}}, \ x_{l_{2}}, \ ...\ ,x_{l_{p}}$ form the system (\ref{3.8}). Therefore
 the solving of the system (\ref{3.6}) is the   solving of such  systems as (\ref{3.8}) and the number of such systems is equal to the number of independent cycles of $ \tau^{-1}\circ \pi$.

The following theorem  fully describes absolute nilpotent elements of $ E^{n}_{\pi,\tau }$.

  \begin{thm}
    Let $x^{*}=(x^{*}_{1},x^{*}_{2},...,x^{*}_{n})$ be an absolute nilpotent element for $ E^{n}_{\pi,\tau }$ and
  $(l_{1}\ l_{2}\ ...\ l_{p})$ be one of the independent cycles of $\tau^{-1}\circ \pi ,\ (p\leq n).$
    \begin {itemize}
   \item [1)] If $a_{l_{k}\pi(l_{k})}\cdot a_{l_{k+1}\tau (l_{k+1})}\neq 0, $ for all $k=1,2,...,p,$ then $x^{*}_{l_{1}}=x^{*}_{l_{2}}=...=x^{*}_{l_{p}}=0$ or
 $\prod \limits_{i=1}^{p}x^{*}_{l_{i}}\neq 0 $ and in this case:
  \item [(1.a)] If  $a_{l_{k_{0}}\pi(l_{k_{0}})} \cdot a_{l_{k_{0}+1}\tau(l_{k_{0}+1})} >0 $ for some $k_{0}=1,...,p$
 then $x^{*}_{l_{1}}=x^{*}_{l_{2}}=...=x^{*}_{l_{p}}=0$, where $l_{p+1}=l_{1}$;

 \item[(1.b)] If $a_{l_{k}\pi(l_{k})} \cdot a_{l_{k+1}\tau(l_{k+1})} <0 $ for all $k=1,...,p $ and  $(-1)^{p}\prod \limits_{i=1}^{p}a_{l_{i}\pi(l_{i})}\neq \prod \limits_{i=1}^{p}a_{l_{i}\tau(l_{i})} $ then $x^{*}_{l_{1}}=x^{*}_{l_{2}}=...=x^{*}_{l_{p}}=0$;

 \item[(1.c)] If $a_{l_{k}\pi(l_{k})} \cdot a_{l_{k+1}\tau(l_{k+1})} <0 $ for all $k=1,...,p $ and  $(-1)^{p}\prod \limits_{i=1}^{p}a_{l_{i}\pi(l_{i})}= \prod \limits_{i=1}^{p}a_{l_{i}\tau(l_{i})} $ then
$$ |x^{*}_{l_{k}}|=\sqrt{\frac{(-1)^{k-1}a_{l_{1}\pi(l_{1})}a_{l_{2}\pi(l_{2})}\cdot ... \cdot a_{l_{k-1}\pi(l_{k-1})}}{a_{l_{2}\tau (l_{2})}a_{l_{3}\tau (l_{3})}\cdot ... \cdot a_{l_{k}\tau (l_{k})}}} \cdot |x^{*}_{l_{1}}| ,\ \ k=2,...,p,$$
where $x^{*}_{l_{1}}$ is any real number.

 \item[2)]  If  $a_{l_{k_{0}}\pi(l_{k_{0}})}=0$ (or $a_{l_{k_{0}}\tau(l_{k_{0}})}=0$ ) for some $k_{0}$  $( 1 \leq k_{0} \leq p)$  and $a_{l_{k}\tau (l_{k})}\neq 0 $ ( or $a_{l_{k}\pi (l_{k})}\neq 0 $ ) for any $k=1,...,p$  then $x^{*}_{l_{1}}=x^{*}_{l_{2}}=...=x^{*}_{l_{p}}=0$.

 \item[3)] If $a_{l_{k_{0}}\pi(l_{k_{0}})}=0$, $a_{l_{k_{0}}\tau(l_{k_{0}})}=0$ for some $k_{0}$ and  $a_{l_{k}\pi(l_{k})}\neq 0 $ or $a_{l_{k}\tau (l_{k})}\neq 0 $ for any $k\neq k_{0}$  then $x^{*}_{l_{k_{0}}}$ will be any real number and $x^{*}_{l_{k}}=0 $ for all $ k\neq k_{0} , \ k=1,...,p$. In this case the number of free parameters is equal to the number of such $k_{0}$.

 \item[4)]  In all other cases either $x^{*}_{l_{1}}=x^{*}_{l_{2}}=...=x^{*}_{l_{p}}=0$ or some of them are equal to zero and the rest depends on the free parameters (the number of free parameters is greater than or equal to one).
\end {itemize}
\end{thm}
  \begin{proof} \
  (1.a)\ If $a_{l_{k}\pi(l_{k})}\cdot a_{l_{k+1}\tau (l_{k+1})}\neq 0 $ for all $k=1,2,...,p $ and $a_{l_{k_{0}}\pi(l_{k_{0}})} \cdot a_{l_{k_{0}+1}\tau(l_{k_{0}+1})} >0 $ for some $k_{0}=1,...,p$ then $ x^{*}_{l_{k_{0}}}=x^{*}_{l_{k_{0}+1}}=0$. In this case  it is easy to see that if
  $ x^{*}_{l_{k_{0}}}=x^{*}_{l_{k_{0}+1}}=0$ then $x^{*}_{l_{1}}=x^{*}_{l_{2}}=...=x^{*}_{l_{p}}=0$.

\

  (1.b),(1.c)\ If  $a_{l_{k}\pi(l_{k})} \cdot a_{l_{k+1}\tau(l_{k+1})} <0$ for all $k=1,...,p $ then from the system (\ref{3.8}) we have
$$ |x_{l_{2}}|=  \sqrt{-\frac{a_{l_{1}\pi(l_{1})}}{a_{l_{2}\tau(l_{2})}}} \cdot |x_{l_{1}}| ,$$
$$ |x_{l_{3}}|=  \sqrt{\frac{a_{l_{1}\pi(l_{1})} \cdot a_{l_{2}\pi(l_{2})}}{a_{l_{2}\tau(l_{2})}a_{l_{3}\tau(l_{3})}}} \cdot |x_{l_{1}}|, $$ $$  ..............................................$$
$$ |x_{l_{p}}|=\sqrt{\frac{(-1)^{p-1}a_{l_{1}\pi(l_{1})}a_{l_{2}\pi(l_{2})}\cdot ... \cdot a_{l_{p-1}\pi(l_{p-1})}}{a_{l_{2}\tau (l_{2})}a_{l_{3}\tau (l_{3})}\cdot ... \cdot a_{l_{p}\tau (l_{p})}}} \cdot |x_{l_{1}}| $$
  \begin{equation}\label{3.9}
 \frac{(-1)^{p-1}a_{l_{1}\pi(l_{1})}a_{l_{2}\pi(l_{2})}\cdot ... \cdot a_{l_{p}\pi(l_{p})} + a_{l_{1}\tau (l_{1})}a_{l_{2}\tau (l_{2})}\cdot ... \cdot a_{l_{p}\tau (l_{p})}}
{a_{l_{2}\tau (l_{2})}a_{l_{3}\tau (l_{3})}\cdot ... \cdot a_{l_{p}\tau (l_{p})}}\cdot x^{2}_{l_{1}}=0.
\end{equation}
Note that $\frac{a_{l_{i}\pi(l_{i})}}{a_{l_{i+1}\tau(l_{i+1})}}<0 $ for $i=1,...,p-1$ and the number of such fractions is  $p-1$. It means that
 \begin{equation}
 \frac{(-1)^{p-1}a_{l_{1}\pi(l_{1})}a_{l_{2}\pi(l_{2})}\cdot ... \cdot a_{l_{p-1}\pi(l_{p-1})}}{a_{l_{2}\tau (l_{2})}a_{l_{3}\tau (l_{3})}\cdot ... \cdot a_{l_{p}\tau (l_{p})}}>0.
 \end{equation}
 From the equation (\ref{3.9}) it follows  that, if
   $$(-1)^{p-1}a_{l_{1}\pi(l_{1})}a_{l_{2}\pi(l_{2})}\cdot ... \cdot a_{l_{p}\pi(l_{p})} + a_{l_{1}\tau (l_{1})}a_{l_{2}\tau (l_{2})}\cdot ... \cdot a_{l_{p}\tau (l_{p})}\neq 0, $$
   i.e $(-1)^{p}\prod \limits_{i=1}^{p}a_{l_{i}\pi(l_{i})}\neq \prod \limits_{i=1}^{p}a_{l_{i}\tau(l_{i})} $  then $x^{*}_{l_{1}}=0$, it follows that  $x^{*}_{l_{1}}=x^{*}_{l_{2}}=...=x^{*}_{l_{p}}=0$. Otherwise,
 if $(-1)^{p}\prod \limits_{i=1}^{p}a_{l_{i}\pi(l_{i})}= \prod \limits_{i=1}^{p}a_{l_{i}\tau(l_{i})} $ then $x^{*}_{l_{1}}$ will be any real number.

\

 (2)  If  $a_{l_{k_{0}}\pi(l_{k_{0}})}=0$ (or $a_{l_{k_{0}}\tau(l_{k_{0}})}=0$ ) for some $k_{0}$  $( 1 \leq k_{0} \leq p)$  and $a_{l_{k}\tau (l_{k})}\neq 0 $ ( or $a_{l_{k}\pi (l_{k})}\neq 0 $ ) for any $k=1,...,p$  then $x^{*}_{l_{k_{0}+1}}=0$ (or $ x^{*}_{l_{k_{0}-1}}=0 $ ). It is known from (3.8) if  $x^{*}_{l_{k_{0}+1}}=0$ (or $ x^{*}_{l_{k_{0}-1}}=0 $ ) then $x^{*}_{l_{1}}=x^{*}_{l_{2}}=...=x^{*}_{l_{p}}=0$.

 \

 (3) If $a_{l_{k_{0}}\pi(l_{k_{0}})}=0$, $a_{l_{k_{0}}\tau(l_{k_{0}})}=0$ for some $k_{0}$ and $a_{l_{k}\pi(l_{k})}\neq 0 $ or $a_{l_{k}\tau (l_{k})}\neq 0 $ for any $k\neq k_{0}$ then it is clearly that $x^{*}_{l_{k_{0}}}$ will be any real number and at least one of $x^{*}_{l_{k_{0}+1}}$ or $ x^{*}_{l_{k_{0}-1}} $ is zero. It follows that
  $x^{*}_{l_{k}}=0 $ for all $ k\neq k_{0} , \ k=1,...,p$.

\

 (4) It is clear that in all other cases the solution of the system (\ref{3.8}) is either trivial or some of them are equal to zero and the rest depends on the free parameters. For that it is enough to consider the next cases :
  \begin {itemize}
 \item [4.1)] Let  $a_{l_{k_{0}}\pi(l_{k_{0}})}=0$, $a_{l_{k_{0}}\tau(l_{k_{0}})}=0$ for some $k_{0}$ and  $a_{l_{k_{1}}\pi(l_{k_{1}})}= 0 $, $a_{l_{k_{2}}\tau (l_{k_{2}})}= 0 $ for some $k_{1}\neq k_{0},k_{2}\neq k_{0}$, in this case:\\
   If  $k_{1}<k_{2}$ and $k_{0}$ doesn't lie between $k_1$ and $k_2$  then $x^{*}_{l_{k_{0}}}$ will be any real number and $x^{*}_{l_{k}}=0 $ for all   $ k\neq k_{0} , \ k=1,...,p$ (solution depends on a free parameter);\\
   If  $  k_{1}<k_{0}<k_{2} $ then  $x^{*}_{l_{k_{0}}} $ will be any real number and $x^{*}_{l_{k}}=0 $ for all $  k_{1}< k <k_{2}, \ k\neq k_{0} $ and  the other $x^{*}_{l_{i}}\ $ depend on a free parameter ( the number of free parameters are greater than or equal to one).\\
   The case $k_{1}>k_{2}$ is similar.
 \item [ 4.2)]  Let  $a_{l_{k}\pi(l_{k)}}=0$, $a_{l_{m}\tau(l_{m})}=0$ for some $k,m=1,...,p ; $   $k<m$, in this case:\\
  If   $a_{l_{k_{0}}\pi(l_{k_{0}})} \cdot a_{l_{k_{0}+1}\tau(l_{k_{0}+1})} >0 $  for some $k_{0} $, which $k_{0}<k $ or $k_{0}>m $ then $x^{*}_{l_{1}}=x^{*}_{l_{2}}=...=x^{*}_{l_{p}}=0$; \\
  If such $k_{0}$  does not exist then $x^{*}_{l_{k+1}}=...=x^{*}_{l_{m-1}}=0 $ and the rest $x^{*}_{l_{i}}\ $  depend on a free parameter.\\
  The case $k>m$ is similar.
  \end {itemize}
\end{proof}

\section{Idempotent elements.}

Now we consider the idempotent elements of $ E^{n}_{\pi,\tau }$.
\begin{defn}
An element $x\in E$ is called $idempotent $ if $x^{2}=x$.
\end{defn}
Such elements of an evolution algebra are especially important, because they are the fixed points of the evolution map $V(x)=x^{2}$, i.e. $V(x)=x$.

 For  $ E^{n}_{\pi,\tau }$ we have $x=\sum\limits_{i}x_{i}e_{i}= \sum\limits_{k}x_{\pi(k)}e_{\pi(k)}$, and the equation $x^{2}=x$ will be (see eq(\ref{3.6}))
\begin{equation}\label{4.1}
a_{k\pi(k)}x^2_{k}+ a_{j_{k}\pi(k)}x^2_{j_{k}}=x_{\pi (k)}, \ \  k \in \{1,...,n\}.
\end{equation}
In  general, the analysis of  solutions of the system (\ref{4.1}) is difficult. Therefore we shall consider some particular solutions of the system (\ref{4.1}):
\begin {itemize}
 \item[ 1)] Trivial solution $(0,...,0);$
 \item[ 2)] Assume that $a_{k\pi(k)} + a_{j_{k}\pi(k)}=d$ for any $k=1,...,n$ , where $d=const,\ d\neq 0,$ then one of the particular solutions
 is $x_{i}=\frac{1}{d}, \ i=1,...,n.$
\end {itemize}

Now we consider the system of equations for $n=2$, then permutations (with $\pi\neq\tau$) will be
 $$
  \pi=\left( \begin{array}{cc} 1  \ \  \ 2 \ \\
 2 \ \ \  1 \ \end{array} \right),\ \ \
    \tau=\left( \begin{array}{cc} 1  \ \  \ 2 \ \\
 1 \ \ \  2 \ \end{array} \right).
 $$
 Then system of equations (\ref{4.1}) will be
 $$\left\{\begin{array}{ll} a_{12}x^{2}_{1}+a_{22}x^{2}_{2}=x_{2} \\

   a_{21}x^{2}_{2}+a_{11}x^{2}_{1}=x_{1}.
   \end{array} \right.$$
 Assume that \ $a_{ij}\neq0 \ ; \ i,j\in\{1,2\}$. For easiness let's denote $a_{12}=a ,\  a_{22}=b ,\ a_{21}=c,\  a_{11}=d $ and $ x_{1}=x,\ x_{2}=y $ then we have \ \ \ \ $$\left\{\begin{array}{ll} ax^{2} + by^{2}=y \\
   dx^{2} + cy^{2}=x
   \end{array} \right.\ \  $$
 from this system we have
  \begin{equation}\label{4.2}
  (bd-ac)^{2}x^{4} - 2b(bd-ac)x^{3}+ (b^{2} + cd)x^{2} - cx = 0
  \end{equation}
 By full analysis of (\ref{4.2}) we get

 \begin{thm}
 Solutions of the equation (\ref{4.2}) are $x=0$ and
\begin {itemize}
\item [1)] If $bd=ac$, $b^{2}+cd\neq 0$ then $ x=\frac{c}{b^{2}+cd}$  .
\item [2)] Let $p=(3cd - b^{2})/3(bd - ac)^{2}, \ q=2(9bcd + b^{3})/27(bd-ac)^{3} - c/(bd - ac)^{2},$
      $\Delta=(q/2)^{2}+ (p/3)^{3}$ if $bd\neq ac$, then
\item [a)] for $\Delta <0$ there exist three real solutions,
\item [b)] for $\Delta >0$ there is one  real solution and two complex conjugate solutions,
\item [c)] for $\Delta =0$ and $p\neq0, q\neq0$  there are two  real solutions,
\item [d)] for $\Delta =0$, and $p=q=0 $  there exist one real solution i.e. $x= 2b/3(bd-ac)$.
\end {itemize}
\end{thm}

\section{ Isomorphism of $ E^{n}_{\pi,\tau }$.}

It is known that any nonidentity permutation of $S_{n}$ can be uniquely expressed (up to the order of the factors ) as a product of disjoint cycles.

\begin{defn} Let $\alpha,\beta \in S_{n}$. Then $\alpha$ and $\beta$ are called conjugate if there exists $\gamma \in S_{n} $ such that  $\gamma \circ \alpha \circ \gamma^{-1}=\beta.$
\end{defn}
\begin{pro}
Let $\delta=(i_1\ i_2\ ...\ i_l)\in S_n$ be a cycle. Then for all $\gamma \in S_n$,
 $\gamma \circ \delta \circ\gamma^{-1}=(\gamma(i_1)\ \gamma(i_2)\ ...\ \gamma(i_l)).$
 \end{pro}
For permutations of the form $\pi=(\pi_{11}\ \pi_{12}\ ... \ \pi_{1k_{1}})\circ (\pi_{21}\ \pi_{22}\ ... \ \pi_{2k_{2}})\circ...\circ (\pi_{r1}\ \pi_{r2}\ ...\ \pi_{rk_{r}})  $ it is known the following result and note that we also consider  cycles which the length is equal to one.
 \begin{pro}
 Two permutations are conjugate if and  only if they have the same cycle type, i.e the corresponding numbers $k_1,k_2,...,k_r$  coincided.
\end{pro}
 Now we consider some particular cases of evolution algebra  $ E^{n}_{\pi,\tau }$. Let $G_\pi$ be set of all conjugate permutations to fixed $\pi$. In the following theorem we consider special type of permutations in $G_\pi$.\\

 \begin{thm}
 Let $ E^{n}_{\pi,\tau }$ be an evolution algebra of permutations with the following conditions:
 \begin {itemize}
\item [ (1)] $a_{i\pi(i)}\cdot a_{i\tau(i)}\neq 0, \ 1\leq i \leq n;$
\item [(2)] $\pi=\pi_{1}\circ \pi_{2}\circ ... \circ \pi_{r}, \ \tau=\tau_{1}\circ \tau_{2}\circ ... \circ\tau_{r},$  where
 $\pi_{1}=(\pi_{11}\ \pi_{12}\ ... \ \pi_{1k_{1}}), $\\
  $ \pi_{2}=(\pi_{21}\ \pi_{22}\ ... \ \pi_{2k_{2}}),\ ... ,\ \pi_{r}=(\pi_{r1}\ \pi_{r2}\ ...\ \pi_{rk_{r}}) $ and $  \tau_{1}=(\tau_{11}\ \tau_{12}\ ... \ \tau_{1k_{1}}),$\\
   $ \tau_{2}=(\tau_{21}\ \tau_{22}\ ...\ \tau_{2k_{2}}), ...,\ \tau_{r}=(\tau_{r1}\ \tau_{r2}\ ...\ \tau_{rk_{r}})$  are  independent  cycles  of   $\pi $ and
   $  \tau  $ respectively,  and $ k_{1}+k_{2}+...+k_{r}=n.$
   \end {itemize}
    If $ \tau_{im}\in \{\pi_{i1}\ \pi_{i2}\ ...\ \pi_{ik_{i}}\},\ 1\leq i \leq r,\ 1\leq m \leq k_{i}\ $
 , i.e.$ \pi_{k}\ $  and $ \tau_{k}$  consist  of  one and  the same  elements, only has
  difference between seats of elements,  \\
 then
\ \ \ \ \ \ \ \ \ \ $$E^{n}_{\pi,\tau }\cong E^{k_{1}}_{\pi_{1},\tau_{1}}\oplus E^{k_{2}}_{\pi_{2},\tau_{2}}\oplus ...\oplus E^{k_{r}}_{\pi_{r},\tau_{r}}. $$
\end{thm}
\begin{proof}
It is easy to see that  $E^{k_{i}}_{\pi_{i},\tau_{i}}$ is $k_{i}$-dimensional evolution algebra of permutations $\pi_{i}$ and $\tau_{i}$ with the basis $ e_{\pi_{im}}\ (or\ e_{\tau_{im}}) $ and the table of multiplications given by
 $$\left\{\begin{array}{ll} e_{\pi_{im}}\cdot e_{\pi_{im}}=a_{\pi_{im}\pi(\pi_{im})}\cdot e_{\pi(\pi_{im})} + a_{\pi_{im}\tau(\pi_{im})}\cdot e_{\tau(\pi_{im})} \\
e_{\pi_{im}}\cdot e_{\pi_{ik}}=0,\ \ m\neq k
   \end{array} \right.\ \  $$
and $\pi(\pi_{im})=\pi_{i,m+1}\in \{\pi_{i1},\pi_{i2},...,\pi_{ik_{i}}\},\ \tau(\pi_{im})=\tau(\tau_{is})\in \{\tau_{i1},\tau_{i2},...,\tau_{ik_{i}}\},\\  1\leq m,s \leq k_{i}.$

The isomorphism is provided by the following change of basis $$e_{i,m}=e_{\pi_{im}}( \ or \ e_{i,m}=e_{\tau_{im}} ). $$
Thus we have the evolution algebra $ E^{k_{i}}_{\pi_{i},\tau_{i}}$ with the basis $e_{i,m},\ 1\leq i \leq r,\ 1\leq m \leq k_{i} $ .

If $i\neq j$ and $m\neq k$ then $\pi_{im}\neq \pi_{jm}$ and  $\pi_{im}\neq \pi_{ik}$  for any $1\leq i,j \leq r,\ 1\leq m,k \leq k_{i} $,
it means
$$ \{ \pi_{i1},\pi_{i2},...,\pi_{ik_{i}} \} \cap \{ \pi_{j1},\pi_{j2},...,\pi_{jk_{j}} \}=\emptyset  .$$ and
$$ \{\pi_{11}\ \pi_{12}\ ... \ \pi_{1k_{1}}\}\cup \{\pi_{21}\ \pi_{22}\ ... \ \pi_{2k_{2}}\}\cup ...\cup \{\pi_{r1}\ \pi_{r2}\ ...\ \pi_{rk_{r}}\}=\{1,2,3,...,n \} .  $$
and we have
$$E^{n}_{\pi,\tau }\cong E^{k_{1}}_{\pi_{1},\tau_{1}}\oplus E^{k_{2}}_{\pi_{2},\tau_{2}}\oplus ...\oplus E^{k_{r}}_{\pi_{r},\tau_{r}}. $$
\end{proof}
Let $ \tau_{0}=(1)(2) ... (n)$ be identity permutation.

\begin{pro}\label{p4}
 Any evolution algebra of permutations $E^{n}_{\pi,\tau_{0} }$ with permutations
$$\pi=(k_{1},k_{2},...,k_{n}),\   and \  a_{i\pi(i)}\cdot a_{ii}\neq 0, \ 1\leq i \leq n $$
is isomorphic to the evolution algebra  $E^{n}_{\pi_{\ast},\tau_{0}} $, with the table of multiplications given by:
$$ \left\{\begin{array}{ll} e'_{i}\cdot e'_{i}= a_{\pi ^{i-1}(1)\pi^{i}(1)}e'_{i+1} + a_{\pi ^{i-1}(1)\pi^{i-1}(1)}e'_{i}, \ \ 1\leq i \leq n \\
   e'_{i}\cdot e'_{j}=0,\ \ i\neq j
   \end{array} \right.\ \  $$
   where  $\pi_{\ast}=(1\ 2\ ...\ n), \ $ and $a_{\pi ^{i-1}(1)\pi^{i}(1)}\cdot a_{\pi ^{i-1}(1)\pi^{i-1}(1)}\neq 0 , \ \pi^{0}(1)=1 $
\end{pro}

\begin{proof}. The isomorphism is established by basis permutations:
$$ e'_{1}=e_{1}, \ e'_{i}=e_{\pi^{i-1}(1)}, \ \ 2\leq i \leq n. $$
Indeed, according to $\tau_{1}(i)=i $ we get

$ e'_{i}\cdot e'_{i}=e_{\pi^{i-1}(1)}\cdot e_{\pi^{i-1}(1)}=a_{\pi ^{i-1}(1)\pi^{i}(1)} \cdot e_{\pi^{i}(1)} + a_{\pi ^{i-1}(1)\tau ( \pi^{i-1}(1))}\cdot e_{\tau ( \pi^{i-1}(1))}= a_{\pi ^{i-1}(1)\pi^{i}(1)} \cdot e_{\pi^{i}(1)} + a_{\pi ^{i-1}(1) \pi^{i-1}(1)}\cdot e_{ \pi^{i-1}(1)}=a_{\pi ^{i-1}(1)\pi^{i}(1)}e'_{i+1} + a_{\pi ^{i-1}(1)\pi^{i-1}(1)}e'_{i}.$
\end{proof}
\

Let's present the next example.

\begin{ex}. Consider the following evolution algebra $A^{n}_{\pi_{\ast},\tau_{0}} $, with the table of multiplications given by:
$$ \left\{\begin{array}{ll} \eta_{i}\cdot \eta_{i}= \eta_{i+1} + \eta_{i}, \ \ 1\leq i \leq n-1 \\
\eta_{n} \cdot \eta_{n}= \eta_{1} + \eta_{n},\\
   \eta_{i}\cdot \eta_{j}=0,\ \ i\neq j
   \end{array} \right.\ \  $$
the algebra  $ A^{n}_{\pi_{\ast},\tau_{0}} $ is evolution algebra of permutations  $\pi_{\ast}=(1\ 2\ ...\ n), \ \tau_{0}$ and all non-zero structural constants $a_{ij}$ are equal to one.
\end{ex}

\begin{thm}
 Any evolution algebra of permutations $E^{n}_{\pi,\tau_{0} }$ with permutations \\
$ \pi=(k_{1},k_{2},...,k_{n}), \tau_{0}=(1)(2) ... (n) $ and condition
 $$ a_{i\pi(i)}\cdot a_{ii}\neq 0, ( a_{\pi^{i-1}(1)\pi^{i-1}(1)})^{2}=a_{\pi^{i-1}(1)\pi^{i}(1)}a_{\pi^{i}(1)\pi^{i}(1)}, \ 1\leq i \leq n $$
 is isomorphic to the algebra $A^{n}_{\pi_{\ast},\tau_{0}} $.
\end{thm}

\begin{proof}
 According to Proposition \ref{p4} it is sufficient to establish isomorphism between evolution algebra $E^{n}_{\pi_{\ast},\tau_{0}} $
 and evolution algebra $A^{n}_{\pi_{\ast},\tau_{0}} $ with  permutations  $\pi_{\ast}=(1\ 2\ ...\ n), \ \tau_{0}$.
 Let the map $f:E^{n}_{\pi_{\ast},\tau_{0}}\rightarrow  A^{n}_{\pi_{\ast},\tau_{0}} $ defined by $f(e'_i)=x_{i}\eta_i$ be isomorphism then
 $$ f(e'_i)f(e'_i)=x^{2}_i \eta_i \eta_i=x^{2}_i(\eta_{i+1} + \eta_i)=x^{2}_i \eta_{i+1}+x^{2}_i \eta_i,$$
 \[f(e'_ie'_i)=f(a_{\pi ^{i-1}(1)\pi^{i}(1)} e'_{i+1} + a_{\pi ^{i-1}(1)\pi^{i-1}(1)}e'_{i}) = \]
 \[=a_{\pi ^{i-1}(1)\pi^{i}(1)}f(e'_{i+1}) +
  a_{\pi ^{i-1}(1)\pi^{i-1}(1)}f(e'_{i})= a_{\pi ^{i-1}(1)\pi^{i}(1)} x_{i+1} \eta_{i+1} + a_{\pi ^{i-1}(1)\pi^{i-1}(1)}x_i \eta_i. \]
 From $ f(e'_i)f(e'_i)=f(e'_ie'_i) $ we have the next system
$$ \left\{\begin{array}{ll} x^{2}_{i}=a_{\pi ^{i-1}(1)\pi^{i}(1)}x_{i+1}\\
x^{2}_{i}=a_{\pi ^{i-1}(1)\pi^{i-1}(1)}x_i
   \end{array} \right.\ \  $$
   $x_i\neq 0$, hence $x_i=a_{\pi ^{i-1}(1)\pi^{i-1}(1)} $ and the equality
    $$( a_{\pi^{i-1}(1)\pi^{i-1}(1)})^{2}=a_{\pi^{i-1}(1)\pi^{i}(1)}a_{\pi^{i}(1)\pi^{i}(1)} $$  is hold  for all  $\ 1\leq i \leq n $.
It means that $ f(e'_i)=a_{\pi ^{i-1}(1)\pi^{i-1}(1)}\eta_i $ is isomorphism.
\end{proof}
\

Let $\alpha,\beta\in S_n$ are conjugated, then there exists $\gamma \in S_{n} $ such that  $\gamma \circ \alpha = \beta \circ\gamma$.

\begin{thm}
 If permutations $\alpha,\beta\in S_n$ are conjugated and $a_{i\alpha(i)}=a_{\gamma(i)\beta(\gamma(i))},\ a_{ii}=a_{\gamma(i)\gamma(i)}, \ 1\leq i\leq n$ then evolution algebras $E^{n}_{\alpha,\tau_0 }$  and  $E^{n}_{\beta,\tau_0 }$  are isomorphic.
 \end{thm}
\begin{proof}
 The map  $f:E^{n}_{\alpha,\tau_0 }\rightarrow E^{n}_{\beta,\tau_0 } $ defined by  $f(e_i)=e_{\gamma(i)}$ is isomorphism. Indeed
$$f(e_ie_i)=f(a_{i\alpha(i)}e_{\alpha(i)}+ a_{i\tau_0(i)}e_{\tau_0(i)})= f(a_{i\alpha(i)}e_{\alpha(i)}+ a_{ii}e_{i})=$$ 
 $$ =a_{i\alpha(i)}f(e_{\alpha(i)})+a_{ii}f(e_{i})=  a_{i\alpha(i)}e_{\gamma(\alpha(i))}+a_{ii}e_{\gamma(i)}, $$ 
$$ f(e_i)f(e_i)= e_{\gamma(i)}e_{\gamma(i)}=a_{\gamma(i)\beta(\gamma(i))}e_{\beta(\gamma(i))}+ a_{\gamma(i)\tau_0(\gamma(i))}e_{\tau_0(\gamma(i))}= $$
$$ =a_{\gamma(i)\beta(\gamma(i))}e_{\beta(\gamma(i))}+ a_{\gamma(i)\gamma(i)}e_{\gamma(i)}. $$
\end{proof}

\begin{pro}
 Any evolution algebra of permutations $E^{n}_{\pi,\tau }$ with permutations
$\pi=(k_{1},k_{2},...,k_{n}),$ $ \tau=(l_{1},l_{2},...,l_{n})$ and $\tau \circ \pi =(1)(2) ... (n)$ is isomorphic to the evolution algebra $E^{n}_{\pi_{\ast},\tau_{\ast}} $,
 which a table of multiplications given by:
  $$ \left\{\begin{array}{ll} e^{'}_{i}\cdot e^{'}_{i}= a_{\pi ^{i-1}(1)\pi^{i}(1)}e^{'}_{i+1} + a_{\pi ^{i-1}(1)\pi^{i-2}(1)}e^{'}_{i-1}, \ \ 1\leq i \leq n \\
   e^{'}_{i}\cdot e^{'}_{j}=0,\ \ i\neq j
   \end{array} \right.\ \  $$
 with permutations $\pi_{\ast}=(1,2,...,n), \ \tau_{\ast}=(1,n,n-1,...,2)$ and $a_{\pi ^{i-1}(1)\pi^{i}(1)}\cdot a_{\pi ^{i-1}(1)\pi^{i-2}(1)}\neq 0. $
\end{pro}

\begin{proof}. Similarly to the proof of Proposition \ref{p4} taking the change of basis
$$ e^{'}_{1}=e_{1}, \ e^{'}_{i}=e_{\pi^{i-1}(1)}, \ \ 2\leq i \leq n. $$ and we have $\tau \circ \pi=(1)(2) ... (n)$ it means that
$\tau(\pi (k))=k,\ \tau(\pi^{i+1}(1))=\tau(\pi(\pi^{i}(1)))=\pi^{i}(1)$ thus we have

$ e^{'}_{i}\cdot e^{'}_{i}=e_{\pi^{i-1}(1)}\cdot e_{\pi^{i-1}(1)}=a_{\pi ^{i-1}(1)\pi^{i}(1)} \cdot e_{\pi^{i}(1)} + a_{\pi ^{i-1}(1)\tau ( \pi^{i-1}(1))}\cdot e_{\tau ( \pi^{i-1}(1))}= a_{\pi ^{i-1}(1)\pi^{i}(1)} \cdot e_{\pi^{i}(1)} + a_{\pi ^{i-1}(1) \pi^{i-2}(1)}\cdot e_{ \pi^{i-2}(1)}=a_{\pi ^{i-1}(1)\pi^{i}(1)}e^{'}_{i+1} + a_{\pi ^{i-1}(1)\pi^{i-2}(1)}e^{'}_{i-1}. $
\end{proof}
\

\section*{Acknowledgements}
The author expresses his deep gratitude to Professor U. A. Rozikov for setting up the problem and for the useful suggestions and to A. Kh. Khudoyberdiyev for helpful suggestions.

\

\end{document}